\documentclass[a4paper,10pt]{article}

\usepackage{mathrsfs,enumerate}
\usepackage{amsmath,amsthm,amssymb}
\usepackage{graphicx,color}
\usepackage{enumitem}

\usepackage{upref}
\usepackage[colorlinks]{hyperref}

\theoremstyle{remark}
\newtheorem{remark}{Remark}
\theoremstyle{definition}
\newtheorem{definition}{Definition}[section]
\theoremstyle{plain}
\newtheorem{proposition}[definition]{Proposition}

\newtheorem{theorem}[definition]{Theorem}
\newtheorem{corollary}[definition]{Corollary}
\newtheorem{lemma}[definition]{Lemma}

 \def\ocirc#1{\ifmmode\setbox0=\hbox{$#1$}\dimen0=\ht0
  \advance\dimen0 by1pt\rlap{\hbox to\wd0{\hss\raise\dimen0
  \hbox{\hskip.2em$\scriptscriptstyle\circ$}\hss}}#1\else {\accent"17 #1}\fi}

\def\et{\quad\text{and}\quad}
\def\bdot{\:.\:}
\def\np{{n+1}}
\def\ip{{i+1}}

\def\ipp{{i+1/2}}
\def\imm{{i-1/2}}
\def\leq{\leqslant}
\def\geq{\geqslant}
\def\R{\mathbb{R}}
\def\N{\mathbb{N}}

\def\Z{\mathbb{Z}}
\def\S{\mathbb{S}}
\def\S{\mathscr{S}}
\def\U{\mathscr{U}}

\def\M{\mathscr{M}}
\def\Ns{\mathscr{N}}
\def\T{\mathscr{T}}

\def\G{\mathscr{G}}

\def\ds{\displaystyle}
\def\epsilon{\varepsilon}
\def\phi{\varphi}
\def\dv{\partial}
\def\gd{\nabla_x}
\def\Dd{\mathrm{D}}
\def\D{\mathscr{D}}
\def\dm{\mathrm{d}}

\def\W{{\rm\bf W}}
\def\L{{\rm\bf L}}

\def\Cz{\mathcal{C}_c^\infty}
\def\Cc{\mathscr{C}}
\def\div{\mathop{{\rm div}_x}}

\def\BV{{\rm {BV}}}

\def\un{\mathbf{1}}
\def\qandq{\quad\text{and}\quad}

\newcommand{\sumi}[0]{\sum_{i=1}^{d}}

\renewcommand{\eqref}[1]{(\ref{#1})}
\newcommand{\eqrefs}[2]{(\ref{#1}-\ref{#2})}

\title{Stability of stationary solutions \\ of singular systems of balance laws}

\author{Nicolas Seguin\footnote{Irmar (UMR 6625), Universit\'e de Rennes 1,
263 avenue du G\'en\'eral Leclerc, CS 74205, 35042 RENNES Cedex, France. (\url{nicolas.seguin@univ-rennes1.fr})}}

\date{}

\begin{document}

\maketitle


\begin{abstract}
The stability of stationary solutions of first-order systems of PDE's are considered. They may include some singular geometric terms, leading to discontinuous flux and non-conservative products. Based on several examples in Fluid Mechanics, we assume that these systems are endowed with a partially convex entropy. We first construct an associated relative entropy which allows to compare two states which share the same geometric data. This way, we are able to prove the stability of some stationary states within entropy weak solutions. This result applies for instance to the shallow-water equations with bathymetry. Besides, this relative entropy can be used to study finite volume schemes which are entropy-stable and well-balanced, and due to the numerical dissipation inherent to these methods, asymptotic stability of discrete stationary solutions is obtained. This analysis does not make us of any specific definition of the non-conservative products, applies to non-strictly hyperbolic systems, and is fully multidimensional with unstructured meshes for the numerical methods.
\end{abstract}


\medskip



\tableofcontents

\section{Introduction}

In this paper, we consider non-conservative systems in $d$ space dimensions of the form
\begin{align}
  \label{eq:sys-ncvu}
  &\dv_t u + \div f(u,\alpha) + \sumi s_i(u,\alpha) \dv_i \alpha = 0 , \\
  \label{eq:sys-ncva}
  &\dv_t \alpha = 0 ,
\end{align}
where $\div=\sumi\dv_i$, $\dv_i$ denotes the partial derivative with respect to $x_i$. We note $f=(f_i)_{i=1,\dots,d}$, and we have
\begin{align*}
 \alpha&\colon \R^+\times\R^d \to \R, & f_i&\colon \Omega\times\R \to \R^N, \\
 u&\colon \R^+\times\R^d \to \Omega, & s_i&\colon \Omega\times\R \to \R^N,
\end{align*}
where $\Omega$ is a convex subset of $\R^N$, the so-called \emph{set of admissible states}. Equation~\eqref{eq:sys-ncva} means that $\alpha$ is time-independent, so that this variable is a data, as soon as an initial condition is associated with \eqrefs{eq:sys-ncvu}{eq:sys-ncva}:
\begin{equation}
  \label{eq:ic} 
  \begin{cases}
    u(0,x) = u_0(x)  \\
    \alpha(0,x) = \alpha(x) 
  \end{cases}
  \text{for } x\in\R^d .
\end{equation}
Therefore, if $\alpha$ is smooth, the third term of the left-hand side of~\eqref{eq:sys-ncvu} can be considered as a source term. However, the analysis of the present paper also applies to non-smooth $\alpha$, and the term $\sum s_i(u,\alpha) \dv_i \alpha$ is a \emph{non-conservative product}. It also applies to the case of systems of conservation laws with discontinuous flux, $f$ being dependent on $\alpha$, as studied for instance in \cite{isaa-temp:nonl-res}, and also in \cite{klingenberg95:_convex} and \cite{AKR} in the scalar case (these two references are only two instances of a huge literature on this subject).

We assume that this system is endowed by an entropy pair $(\eta,F)$, which depends on $(u,\alpha)$ and satisfies the following assumptions:
\begin{enumerate}[label=(H\arabic*)]
\item \label{item:H1} The function $\eta=\eta(u,\alpha)\in\Cc^2(\Omega\times\R,\R)$ is convex with respect to its first variable and there exist two positive constants $\underline\eta<\overline\eta$ such that
\begin{equation}
 \label{eq:eta2}
 \sigma(\dv_u^2 \eta) \subset [\underline\eta,\overline\eta] \text{ on } \Omega\times\R,
\end{equation}
where $\sigma$ denotes the matrix spectrum. 
\item \label{item:H2} There exists an entropy flux $F=(F_i(u,\alpha))_{i=1,\dots,d}$ such that
\begin{equation}
  \forall i=1,\dots,d, \quad \dv_u \eta \, \dv_u f_i = \dv_u
  F_i \text{ and } \dv_u \eta \, (\dv_\alpha f_i+s_i) = \dv_\alpha F_i.
\end{equation}
\end{enumerate}

Since system~\eqrefs{eq:sys-ncvu}{eq:sys-ncva} is non-conservative, the products $s_i\dv_i\alpha$ are not defined for weak solutions, and generalised theories are invoked, see for instance \cite{colombeau} or \cite{dalm-lef-mur:non-cons}. We do not need here to provide a particular definition of weak solutions, we mainly impose that these solutions satisfy the entropy inequality
\begin{equation}
  \label{eq:entncvin}
  \dv_t \eta(u,\alpha) + \div F(u,\alpha) \leq 0 ,
\end{equation}
which becomes an equality for smooth solutions because of~\ref{item:H2}. Since the left hand side of this inequality is in a conservative form, it is well defined for weak solutions.

\begin{remark}
 The convexity assumption~\eqref{eq:eta2} is assumed on the whole space $\Omega\times\R$. It may be restrictive and, in order to deal with more general case, one could restrict the discussion of this paper to some neighborhood of a constant state of $\Omega\times\R$ without any major change.
\end{remark}

The issue addressed in this work is the role of the entropy inequality~\eqref{eq:entncvin} for the stability analysis of non-conservative systems of the form~\eqrefs{eq:sys-ncvu}{eq:sys-ncva}, and more precisely, the nonlinear stability of stationary solutions of~\eqrefs{eq:sys-ncvu}{eq:sys-ncva}. In applications, they are very important since they may serve not only as initial conditions (before being perturbed by a particular event on the domain or by a modification of the boundary conditions), but also as solutions which can be reached in the long time limit. Our aim is two-fold: 
\begin{itemize}
\item determine stationary solutions which are stable,
\item study the stability of numerical schemes when computing these stationary solutions. 
\end{itemize}

As we will see later, several systems of interest enter in our framework: shallow-water equations with bathymetry, gas dynamics in porous media, one-dimensional gas dynamics in a nozzle\dots Let us focus on the first example in this introduction. When the influence of the bathymetry is taken into account, the shallow-water equations write
\begin{equation}
  \label{eq:sw}
  \begin{cases}
    \partial_t h + \div (hv) = 0 \\
    \partial_t (hv) + \div (hv\otimes v) + \gd (gh^2/2) + gh \gd \alpha = 0 \\
    \partial_t \alpha = 0
  \end{cases}
\end{equation}
where $h$ is the height of water, $v$ is the average horizontal velocity, $g$ the gravity constant, and $\alpha$ the altitude of the ground. The most considered stationary states correspond to the ``lake at rest'' case, i.e.
\begin{equation}
  \label{eq:sw-ss}
  \begin{cases}
     \gd (\bar h+\alpha) = 0 \\
     \bar v = 0
  \end{cases}
\end{equation}
where $\bar h$ and $\bar v$ only depend on $x$. We first prove that stationary solutions~\eqref{eq:sw-ss} are (nonlinearly) stable in the class of entropy weak solutions of~\eqref{eq:sw}, the bathymetry $\alpha$ being given. Let us emphasize that this stability holds even for non-smooth bathymetry, $\alpha\in\BV$ for instance, and is independent of the definition of the non-conservative product $h \gd \alpha$. On the other hand, we investigate the behavior of finite volume schemes which satisfy a discrete version of the entropy inequality~\eqref{eq:entncvin} (with $\eta((h,hv),\alpha)=h|v|^2/2+gh^2/2+gh\alpha$ and $F((h,hv),\alpha)=v(\eta+gh^2/2)$). If in addition they are well-balanced with respect to stationary states~\eqref{eq:sw-ss} (i.e. they exactly preserve a discretized version of~\eqref{eq:sw-ss}) then one can deduce the asymptotic stability of these states, due to the numerical diffusion.
Note that the design of well-balanced schemes for shallow-water equations~\eqref{eq:sw} has deserved a huge attention these last twenty years, but in general, the discrete entropy inequalities are difficult to obtain. Examples and references will be provided in the sequel.

Let us emphasize that the present analysis is independent of the space dimension, and of the hyperbolicity of system~\eqrefs{eq:sys-ncvu}{eq:sys-ncva}. In some sense, it is very complementary to the Amadori and Gosse's works, see \cite{AmadoriGosseBook}, and shares the property to lead to estimates which are uniform in time.

The main tool we use to obtain these results is the \emph{relative entropy}. Let us briefly recall this notion in the conservative case.

\paragraph{Relative entropy for systems of conservation laws.}
Consider a $N\times N$ system of conservation laws
\begin{equation}
  \label{eq:scl}
  \dv_t u + \div f(u) = 0
\end{equation}
endowed with a Lax entropy pair $(\eta,F)$, $\eta$ being strictly convex (in a similar sense as in assumption~\ref{item:H1}), i.e. admissible weak solutions of~\eqref{eq:scl} have to satisfy the inequality
\begin{equation*}
  \dv_t \eta(u) + \div F(u) \leq 0,
\end{equation*}
in the weak sense. The relative entropy associated with system~\eqref{eq:scl} is defined by
\begin{equation*}
  h(u,v) = \eta(u) - \eta(v) - \nabla\eta(v) \cdot (u-v) .
\end{equation*}
Note that this function is not symmetric, and we should say that $h$ is the entropy for $u$ relatively to $v$. It is easy to check that
\begin{equation}
  \label{eq:scl-ent-quad}
  \underline\eta |u-v|^2 \leq  h(u,v) \leq \overline\eta |u-v|^2
\end{equation}
where $|\cdot|$ is the Euclidian norm of $\R^N$ and $\sigma(\nabla^2 \eta) \subset [\underline\eta,\overline\eta]$.

Now, let us consider an admissible weak solution $u$ of~\eqref{eq:scl} and a constant vector $v\in\R^N$. After some calculations, one obtains
\begin{equation}
  \label{eq:scl-ent-pde}
  \dv_t h(u,v) + \div \big( F(u) - \nabla\eta(v) \cdot f(u) \big) \leq 0 . 
\end{equation}
If we integrate this inequality for $x\in\R^d$, the divergence term disappears and we have
\begin{equation*}
  \frac{\dm}{\dm t} \int_{\R^d} h(u,v) \ \dm x \leq 0.
\end{equation*}
We then deduce from~\eqref{eq:scl-ent-quad} the $L^2$-stability of constant, and thus stationary, solutions $v$ in the class of admissible weak solutions.

\begin{remark}
  In \cite{dafermos:entrel} and \cite{DiP:uniq}, Dafermos and DiPerna respectively proved such a stability result when $v$ is a strong solution of~\eqref{eq:scl}, also referred as weak--strong uniqueness. Note that the set of entropy weak solutions has been enlarged to measured-valued solutions in \cite{BDLS}. We believe that the present work could also be extended to this framework.
\end{remark}

\paragraph{Outline of the paper.} 

In this work, we extend the previous analysis to systems of the form~\eqrefs{eq:sys-ncvu}{eq:sys-ncva}. For a given $\alpha$, we are able to compare an entropy weak solution $u$ to some particular stationary solutions. In section~\ref{sec:cont}, we detail the class of admissible weak solutions of \eqrefs{eq:sys-ncvu}{eq:sys-ncva} we consider in this work, which does not use any explicit definition of the non-conservative term. We then state and prove Theorem~\ref{thm:stab}, on the nonlinear stability of particular stationary states of \eqrefs{eq:sys-ncvu}{eq:sys-ncva}. In the next section, we provide some examples of systems which enter in this framework, and we explicit the associated stable stationary states. The aim of section~\ref{sec:num} is to present the discrete case. We then focus on entropy-stable finite volume schemes, which are well-balanced at least for the stationary states which are nonlinearly stable. Due to the numerical diffusion of first order time explicit schemes, the discrete stationary states are asymptotically stable. 

\section{Stability of stationary solutions}
\label{sec:cont}

\subsection{Definition of weak solutions}

We aim at proving that stationary solutions are stable among entropy weak solutions. However, since we consider discontinuous $\alpha$, only in $\BV$ for instance, the products $s_i\dv_i\alpha$ in \eqref{eq:sys-ncvu} are not defined. Several theories exist in the literature to define them, but here we only use some basic and natural assumptions. We assume that the products $s_i\dv_i\alpha$ can be described by means of vector-valued Radon measures $\mu_i\in\M(\R^+\times\R^d)^N$\footnote{More precisely, $\M(X)$ denotes the set of locally bounded Radon measures on a set $X$, i.e. $\M(X)=(\Cc_\mathrm{c}(X))'$.} which satisfy at least the following properties:  
\begin{enumerate}[label=(P\arabic*)]
\item \label{item:P1} On any open set $B=B_t\times B_x\subset\R^+\times\R^d$ such that $\alpha\in\W^{1,\infty}(B_x)$, the measures $\mu_i$, $i=1,\dots,d$, satisfy
\begin{equation*}
 \forall \phi\in\Cz(B), \forall i=1,\dots,d, \ \int_B \phi \ \dm\mu_i(t,x) = \int_B \phi s_i(u,\alpha) \dv_i \alpha \ \dm t \ \dm x .
\end{equation*}
\item \label{item:P2} For any component $k=1,\dots,N$ and any dimension index $i=1,\dots,d$, 
\begin{equation*}
s_i^{(k)}\equiv0 \quad \Longrightarrow \quad \mu_i^{(k)}\equiv0 .
\end{equation*}
\end{enumerate}
We are now in position to provide a very general definition of solution.
\begin{definition}
 Let $u_0\in\BV(\R^d,\Omega)^N$, $\alpha\in\BV(\R^d)$ and $T>0$. A function $u\in\Cc((0,T);\BV(\R^d,\Omega))$ is an entropy weak solution of the Cauchy problem~\eqrefs{eq:sys-ncvu}{eq:sys-ncva}--\eqref{eq:ic} if there exists $(\mu_i)_{1\leq i\leq d}\subset\M(\R^+\times\R^d)$ satisfying assumptions~\ref{item:P1} and \ref{item:P2} such that, for all $\phi\in\Cz([0,T)\times\R^d)$, 
\begin{multline}
 \label{def:weak}
 - \int_0^T \int_{\R^d} \bigg( u \dv_t \phi + \sumi f_i(u,\alpha) \dv_i\phi \bigg) \ \dm x \ \dm t + \int_0^T \int_{\R^d} \phi \ \dm \mu(t,x) \\
 - \int_{\R^d} u_0(x)\phi(0,x) \ \dm x = 0 ,
\end{multline}
and, for all nonnegative $\phi\in\Cz([0,T)\times\R^d)$,
\begin{multline}
 \label{def:weak-ent}
 \int_0^T \int_{\R^d} \bigg( \eta(u,\alpha) \dv_t \phi + \sumi F_i(u,\alpha) \dv_i \phi \bigg) \ \dm x \ \dm t \\
 + \int_{\R^d} \eta(u_0,\alpha)(x)\phi(0,x) \ \dm x \geq 0 .
\end{multline}
\end{definition}
Such a definition is not sufficient to hope a well-posedness result, without any additional assumption on the measures $\mu_i$, but it is sufficient to obtain the stability results of the next sections. Besides, assumption~\ref{item:P1} is not necessary for the upcoming analysis. We introduce it to ensure that, if $\alpha\in\W^{1,\infty}(\R^d)$, the standard definition of entropy weak solutions is recovered.
It is also important to note that inequalities~\eqref{def:weak-ent} exactly correspond to the weak form of~\eqref{eq:entncvin}, so that the measures $\mu_i$ do not appear there. 

\subsection{Relative entropy and nonlinear stability}

As mention in the introduction, it seems impossible to construct a relative entropy for system~\eqrefs{eq:sys-ncvu}{eq:sys-ncva} to compare two solutions $(u,\alpha)$ and $(v,\beta)$. Nonetheless, one can define a relative entropy between two solutions $u$ and $v$, $\alpha$ being given and common.

\begin{definition}
 \label{def:er}
 The \emph{relative entropy} associated with the non-conservative system~\eqrefs{eq:sys-ncvu}{eq:sys-ncva}, endowed with an entropy $\eta$, is
\begin{equation}
 \begin{aligned}
  h \colon \Omega\times\Omega\times\R &\longrightarrow \R^+ \\
  	(u,v,\alpha) &\longmapsto \eta(u,\alpha) - \eta(v,\alpha) - \dv_u \eta(v,\alpha) \cdot (u-v) .
 \end{aligned} 
\end{equation}
\end{definition}

\begin{lemma}
\label{lem:h}
Assume that the entropy $\eta$ satisfies~\ref{item:H1}. Then, the relative entropy is convex with respect to its first variable and for all $u,v\in\Omega$, we have
\begin{equation}
 \underline\eta |u-v|^2 \leq h(u,v,\cdot) \leq \overline\eta |u-v|^2 .
\end{equation}
\end{lemma}

For a given $\alpha\in\Cc^1(\R^d)$, consider a smooth, and thus entropy conservative, solution $u$ of \eqrefs{eq:sys-ncvu}{eq:sys-ncva}, and a time-independent function $v$. Let us compute the equation satisfied by the relative entropy $h$:
\begin{align*}
 \dv_t h(u,v,\alpha) &= \dv_t \eta(u,\alpha) - \dv_u \eta(v,\alpha) \cdot \dv_t u \\
 &= - \div F(u,\alpha) + \dv_u \eta(v,\alpha) \cdot \sumi \big( f_i(u,\alpha) +  s_i(u,\alpha) \dv_i \alpha \big) \\
 &= - \div \big( F(u,\alpha) - \dv_u\eta(v,\alpha) \cdot f(u,\alpha) \big) \\
 &\quad - \sumi \dv_i(\dv_u \eta(v,\alpha)) \cdot f_i(u,\alpha) + \dv_u \eta(v,\alpha) \cdot \sumi s_i(u,\alpha) \dv_i \alpha .
\end{align*}
The two last terms are not in conservative form, but one could make them vanishing adding some assumptions on $v$. To do so, for any given constant vector $H_0\in\R^N$, we introduce $\S(H_0)$, the set of $(v,\alpha)\in\Omega\times\R$ such that:
\begin{enumerate}[label=(S\arabic*)]
\item \label{item:S1} $\dv_u\eta(v,\alpha) = H_0$.
\item \label{item:S2} For all $i=1,\dots,d$ and $k=1,\dots,N$, $H_0^{(k)} s_i^{(k)} \equiv 0$.
\end{enumerate}
We are then able to state the following stability result:
\begin{theorem}
  \label{thm:stab} Let $H_0\in\R^N$ and consider the set $\S(H_0)$ defined by \ref{item:S1} and \ref{item:S2}, assumed to be nonempty. Consider $\alpha\in\BV(\R^d)$ and a function $v\in\BV(\R^d,\Omega)$ such that $(v,\alpha)\in\S(H_0)$ almost everywhere. Then, $v$ is a stationary entropy weak solution of system~\eqrefs{eq:sys-ncvu}{eq:sys-ncva}. \\
  Moreover, let $T>0$, $u_0\in\BV(\R^d,\Omega)^N$, and $u\in\Cc((0,T);\BV(\R^d,\Omega))$ an associated entropy weak solution. Then, there exists a positive constant $L_f$, independent of $u$, $v$ and $\alpha$ such that
the following nonlinear stability property holds for all $R>0$ and for almost every $t\in[0,T]$:
\begin{equation}
 \label{eq:stabt}
 \int_{B(0,R)} h(u(t,x),v(x),\alpha(x)) \ \dm x \leq \int_{B(0,R+L_ft)} h(u_0(x),v(x),\alpha(x)) \ \dm x .
\end{equation}
\end{theorem}

\begin{proof}
 First, let us remark that the stability inequality~\eqref{eq:stabt} implies that $v$ is a stationary solution of system~\eqrefs{eq:sys-ncvu}{eq:sys-ncva}. Indeed, if we choose $u_0=v$, the right-hand side of~\eqref{eq:stabt} is null, by the properties of $h$, see Lemma~\ref{lem:h}. Therefore, $u$ being an entropy weak solution and $v$ being time-independent, one may deduce that $v$ is a stationary entropy weak solution using once again the properties of $h$.
 
 Let us now rewrite the calculations described above, but in the weak sense. By assumptions~\ref{item:S1} and~\ref{item:S2}, $H_0\cdot s_i(\cdot,\alpha)=0$ for all $i$. In other words, this means that if the $i$-th component of $H_0$ is non-zero, then $s_i\equiv0$. We now use the definition of $u$ and assumption~\ref{item:P2} on the non-conservative product to obtain, for all $\phi\in\Cz([0,T)\times\R^d)$, 
\begin{equation*}
 \label{def:Hweak}
 \int_0^T \int_{\R^d} H_0\cdot(u \dv_t \phi + \sumi f_i(u,\alpha) \dv_i \phi) \ \dm x \ \dm t + \int_{\R^d} H_0\cdot u_0(x)\phi(0,x) \ \dm x = 0 .
\end{equation*}
Now, using the entropy inequality \eqref{def:weak-ent} for $u$ and the fact that $v$ is independent of time, one has
\begin{multline}
\label{eq:hedp}
 \int_0^T \int_{\R^d} h(u,v,\alpha) \dv_t \phi \ \dm x \ \dm t 
 + \int_0^T \int_{\R^d} \sumi \big( F_i(u,\alpha) - H_0 \cdot f_i(u,\alpha) \big) \dv_i \phi \ \dm x \ \dm t \\
 + \int_{\R^d} h(u_0,v,\alpha) \phi(0,x) \ \dm x \geq 0 ,
\end{multline}
since $h(u,v,\alpha)=\eta(u,\alpha)-\eta(v,\alpha)-H_0\cdot(u-v)$. To obtain inequality \eqref{eq:stabt}, we introduce $L_f$ such that
\begin{equation*}
 |F-H_0 \cdot f| \leq L_f h
\end{equation*}
which is comparable to the maximum of the spectral radii of $\dv_uf_i$ (for more details, see \cite{dafermos3} and \cite{CMS}). It suffices now to introduce, $t$ and $R$ being fixed, 
\begin{equation*}
 w_\epsilon(\tau) = 
 \begin{cases}
  1 & 0 \leq \tau \leq t \\
  1 + (t-\tau)/\epsilon & t < \tau \leq t+\epsilon \\
  0 & t+\epsilon < \tau
 \end{cases}
\end{equation*}
and
\begin{equation*}
 \chi_\epsilon(\tau,x) = 
 \begin{cases}
  1 & |x| \leq R + L_f (t-\tau) \\
  1 + (R + L_f (t-\tau) - |x|)/\epsilon & 0 < |x| - R - L_f (t-\tau) \leq \epsilon \\
  0 & R + L_f (t-\tau) + \epsilon < |x|
 \end{cases}
\end{equation*}
and take $\phi(\tau,x)=\chi_\epsilon(\tau,x)w_\epsilon(\tau)$ (we omit the passage from Lipschitz continuous functions to $\Cz$ functions). Plugging this test function in \eqref{eq:hedp} yields
\begin{align*}
& \frac{1}{\epsilon} \int_t^{t+\epsilon} \int_{B(0,R+\epsilon)} h(u,v,\alpha)(\tau,x) \chi_\epsilon(\tau,x) \ \dm x \ \dm \tau \\
& \leq \int_{B(0,R+L_ft+\epsilon)} h(u_0,v,\alpha) \chi_\epsilon(0,x) \ \dm x \\
& - \frac{1}{\epsilon} \int_0^{t+\epsilon} \int_{0 < |x| - R - L_f (t-\tau) < \epsilon} w_\epsilon(\tau) \bigg[ L_f h(u,v,\alpha)  \\
& \qquad \qquad \qquad \qquad \qquad + \frac{x}{|x|} ( F(u,\alpha)- H_0 \cdot f(u,\alpha)) \bigg] \ \dm x \ \dm \tau .
\end{align*}
By definition of $L_f$, the last integral is nonnegative, so that,  letting $\epsilon$ tend to $0$ provides inequality \eqref{eq:stabt}.
\end{proof}

\begin{remark}
 As mentioned above, assumption~\ref{item:P1} has not been used in the proof. 
\end{remark}
\begin{remark}
 The nonlinear stability due to \eqref{eq:stabt} implies the $\L^2$ stability and the uniqueness of stationary solutions $v$ satisfying \ref{item:S1} and \ref{item:S2}, in the class of entropy weak solutions. Let us stress that $v$ is only $\BV$, while classical results of nonlinear stability are obtained assuming the smoothness of the reference solution. Moreover, this result is independent of the possible lack of hyperbolicity of the system, and it turns out that it applies to systems for which uniqueness may fail, such as those considered for instance in \cite{isaacsontemple95} and \cite{goatinlfres} (see also \cite{chinnayya04} for the particular case of shallow-water equations).
\end{remark}
\begin{remark}
 We are not able to extend estimate \eqref{eq:stabt} to non-stationary solution $v$, keeping $\alpha$ in $\BV(\R^d)$. The case of a smooth $\alpha$ could probably be handled by adapting the results of \cite{JR:ER}.
\end{remark}
\begin{remark}
\label{rem:stab}
 We only obtain the stability of $v$, but asymptotic stability could not be reached without more structure. Indeed, even for standard systems of conservation laws, asymptotic stability of constant solutions is merely proved for genuinely nonlinear $2\times2$ systems of conservation laws \cite{glimm-lax} (see \cite{dafermos3} for more references). Another way to obtain asymptotic stability would be to add some dissipative term to system~\eqrefs{eq:sys-ncvu}{eq:sys-ncva}. On the other side, the existence of time periodic non-dissipative solutions have been addressed in a series of paper by Temple and Young, the most recent being \cite{MR3377838}. 
\end{remark}

\section{Examples}
\label{sec:ex}

We provide here some examples of equations which enter in this framework. In each case, we provide the stationary solutions concerned by Theorem~\ref{thm:stab}.

\subsection{Shallow-water equations with bathymetry}

The first example is the well-known Saint-Venant system, which models a free surface flow of water of a non flat bottom. The unknowns are the height of water $h$, assumed to remain positive, and the depth-averaged velocity $U$. They satisfy the following equations, posed for $(t,x)\in\R^+\times\R^2$:
\begin{equation}
\label{eq:stv}
\begin{cases}
 \dv_t h + \div (hU) = 0 , \\
 \dv_t (hU) + \div (h U \otimes U) + \gd \bigg(g \ds\frac{h^2}{2} \bigg) + gh \gd \alpha = 0 , \\
 \dv_t \alpha = 0 .
\end{cases}
\end{equation}
Here, $\alpha$ plays the role of the bathymetry, and $g$ is the gravity constant. By its simplicity in comparison with the incompressible Euler equations with a free surface, this model is very popular and numerical simulations show its reliability, even when $\alpha$ is discontinuous.

This system of equations may be endowed with an entropy inequality of the form~\eqref{eq:entncvin}, setting
\begin{equation*}
 \eta(u,\alpha) = h U^2/2 + gh(h/2+\alpha) \et F_i(u,\alpha) = U_i(\eta(u,\alpha)+gh^2/2)
\end{equation*}
where $u=(h,hU)$. The convexity of $\eta$ with respect to $u$ is classical and one can see that $\eta$ is only linear in $\alpha$.

The description of all possible stationary solutions is very difficult in practice. The simplest ones correspond to a ``lake at rest" and are defined by
\begin{equation}
 \label{eq:lake}
 h+\alpha = Z_0 \et U=0 \quad \text{a.e.}
\end{equation}
where $Z_0$ is a given real constant greater than the maximum of $\alpha$. On the other hand, the entropy variable for is
\begin{equation*}
 \dv_u \eta(u,\alpha) = 
\begin{pmatrix}
 - U^2/2 + g(h+\alpha) \\ U
\end{pmatrix} .
\end{equation*}
As a consequence, assumption~\ref{item:S2} yields $U=0$, since $s_1=s_2=(0,gh)^\top$. Next, assumption~\ref{item:S1} corresponds to equality $h+\alpha=Z_0$. To sum up, we have:
\begin{corollary}
 Stationary states of the shallow-water equations~\eqref{eq:stv} given by~\eqref{eq:lake} (lake at rest) are nonlinearly stable, in the sense of theorem~\ref{thm:stab}.
\end{corollary}


\subsection{Gas dynamics in porous media}

We now study the compressible Euler equations. When the flow lies in porous media, $\alpha>0$ being the porosity, they become
\begin{equation}
 \label{eq:Euler}
\begin{cases}
 \dv_t (\alpha \rho) + \div (\alpha \rho U) = 0, \\
 \dv_t (\alpha \rho U) + \div (\alpha \rho U \otimes U) + \gd (\alpha p) - p \gd \alpha = 0, \\
 \dv_t (\alpha \rho E) + \div (\alpha U(\rho E + p)) = 0,
\end{cases}
\end{equation}
where $\rho$, $U$, $E$ and $p$ respectively are the density, the velocity, the total energy and the pressure of the fluid. Note that the porosity $\alpha$ is supposed to be positive. The total energy is the sum of the specific energy and the kinetic energy, i.e.
\begin{equation*}
 E = e + U^2/2
\end{equation*}
and we assume the fundamental thermodynamic relation
\begin{equation}
\label{eq:Tds}
 T \dm s = \dm e + p \, \dm \tau,
\end{equation}
where $s$ is the specific entropy, $T$ the temperature and $\tau=1/\rho$ the specific volume. By classical calculations, one may check that classical solutions of~\eqref{eq:Euler} satisfy
\begin{equation*}
 \dv_t s + U \bdot \gd s = 0 .
\end{equation*}
From this equation and mass conservation, we deduce the following entropy inequality for weak solutions,
\begin{equation*}
 \dv_t (- \alpha\rho s) + \div (- \alpha\rho s U) \leq 0.
\end{equation*}
On the other hand, the function $S\colon(\rho,\rho U,\rho E)\mapsto-\rho s(\tau,e)$ is convex if $s$ is concave, see for instance \cite{godlewski96}. As a consequence, if we note $u=(\alpha\rho,\alpha\rho U,\alpha\rho E)$, the function $\eta\colon(u,\alpha)\mapsto-\alpha\rho s(\tau,e)$ is also convex w.r.t. $u$ by the identity $\eta(u,\alpha)=\alpha S(u/\alpha)$, while the associated entropy flux is $F(u,\alpha)=U\eta(u,\alpha)$.

By classical calculations, one may check that
\begin{equation*}
\partial_u \eta(u,\alpha) = \frac{1}{T}
\begin{pmatrix}
e+p/\rho-Ts-|U|^2/2 \\
U \\
-1
\end{pmatrix} .
\end{equation*}
Due to the form of system~\eqref{eq:Euler}, assumption \ref{item:S1} on the third component of $\dv_u\eta(u,\alpha)$ leads to a constant temperature $T$. Since $T>0$, assumption \ref{item:S2} implies $U=0$. At last, by classical thermodynamical arguments, assumption \ref{item:S1} on the first component of $\dv_u\eta(u,\alpha)$ provides that $p$ is also constant.

\begin{corollary}
 Stationary states of the compressible Euler equations in porous media~\eqref{eq:Euler} such that the temperature $T$ and the pressure $p$ are constant and $U=0$ a.e. are nonlinearly stable, in the sense of theorem~\ref{thm:stab}.
\end{corollary}

\subsection{Gas dynamics with sources in Lagrangian coordinates}


In the previous example, the source term leads to a standing wave in Eulerian coordinates where $\alpha$ is discontinuous. In the case of some source terms, like gravity, it may be more relevant to use a moving wave associated to $\alpha$. This enables in particular to construct well-balanced and asymptotic preserving schemes, see \cite{cargo94}, \cite{gallice:crassource} or \cite{gdtCEA:AS} for instance. In order to deal with this case, we place ourselves in Lagrangian coordinates and study the problem 
\begin{equation}
\label{eq:lag}
\begin{cases}
 \Dd_t \alpha = 0, \\
 \Dd_t \tau - \dv_m U = 0, \\
 \Dd_t U + \dv_m p = \dv_m \alpha, \\
 \Dd_t E + \dv_m (p U) = U \dv_m \alpha,
\end{cases}
\end{equation}
where $u=(\tau,U,E)$, while $\tau$, $p$ and $E$ are the same variables as in the previous section. In the case of a gravity source term, one defines $\dv_m\alpha=g$. Multiplying the PDE's for $U$ and subtracting it to the PDE for $E$ provides
\begin{equation*}
 \Dd_t e + p \dv_m U = 0 .
\end{equation*}
This PDE, together with the fundamental thermodynamic relation∞ \eqref{eq:Tds}, gives
\begin{equation*}
 \Dd_t s = 0 .
\end{equation*}
One may remark that $s$ is independent of $\alpha$ and is a concave function, so that it suffices to choose $\eta=-s$ for the mathematical entropy. From assumption~\ref{item:S1}, one deduces that the stationary states of interest are constant states. Unfortunately, assumption~\ref{item:S2} cannot be fulfilled, because of the third component which would lead to stationary states with an infinite temperature. In order to circumvent this, we follow \cite{cargo94} where was remarked that
\begin{equation*}
 U \dv_m \alpha = \dv_m (U \alpha) - \alpha \dv_m U = \dv_m (U \alpha) - \alpha \Dd_t \tau = \dv_m (U \alpha) - \Dd_t (\alpha\tau) .
\end{equation*}
By introducing $F=E+\tau\alpha$, system~\eqref{eq:lag} becomes conservative:
\begin{equation}
\label{eq:lagF}
\begin{cases}
 \Dd_t \alpha = 0, \\
 \Dd_t \tau - \dv_m U = 0, \\
 \Dd_t U + \dv_m (p-\alpha) = 0, \\
 \Dd_t F + \dv_m ((p-\alpha) U) = 0.
\end{cases}
\end{equation}
Using the new variable $u=(\tau,u,F)$, we define the mathematical entropy 
\begin{equation*}
 \eta(u,\alpha) = - s(\tau,F-\tau\alpha-U^2/2) ,
\end{equation*}
which satisfies assumptions \ref{item:H1} and \ref{item:H2}. As far as stationary states are concerned, condition \ref{item:S2} is trivial since system~\eqref{eq:lagF} is conservative, while condition \ref{item:S1} leads to the classical hydrostatic equilibria.

\begin{corollary}
 Stationary states of the compressible Euler equations with gravity in Lagrangian coordinates~\eqref{eq:lagF} such that the velocity $U$ and $p-\alpha$ are constant are nonlinearly stable, in the sense of theorem~\ref{thm:stab}.
\end{corollary}

\section{Well-balanced schemes and stability of discrete stationary solutions}
\label{sec:num}

We look now at the discrete counterpart of the nonlinear stability result stated in theorem~\ref{thm:stab}. We focus here in time-explicit first order finite volume methods, but the most of the following analysis can be easily extended to more complex methods which share similar properties.

\subsection{General setting and main properties of the schemes}

We consider a general mesh of $\R^d$, denoted by $\T$, defined as a family of disjoint polygonal connected subsets of $\R^d$ such that $\R^d$ is the union of the closure of the elements of $\T$ (called control volumes or cells in the following) and such that the common “interface” of any two control volumes is included in a hyperplane of $\R^d$. The interface which separates two control volumes $K$ and $L$  is noted $e_{KL}$ (we have of course $e_{KL}=e_{LK}$) and $n_{KL}$ the unit normal vector to $e_{KL}$ oriented from $K$ to $L$ (then $n_{KL}=-n_{LK}$). For any $K\in\T$, $\Ns(K)\subset\T$ denotes the set of cells which share a common interface with $K$. We assume that $h=\sup\{\mathrm{diam}(K),K\in\T\}<+\infty$ and that there exists a constant $a>0$ such that
\begin{equation}
\label{eq:meshreg}
\forall K\in\T, \quad |K| \geq ah^d \qandq \partial K = \sum_{L\in \Ns(K)} |e_{KL}| \leq \frac{h^{d-1}}{a} ,
\end{equation}
where $|K|$ is the $d$-dimensional Lebesgue measure of $K$ and $|e_{KL}|$ the $(d-1)$-dimensional Lebesgue measure of $e_{KL}$. For simplicity, we use a uniform time discretization, introducing the time step $\Delta t>0$.

The discrete sequence of approximation is denoted $(u_K^n)$, with $n\in\N^*$ and $K\in\T$, where the initial data is given by
\begin{equation}
 \label{eq:icFV}
 u_K^0 = \frac{1}{|K|} \int_K u_0(x) \ \dm x,
\end{equation}
and the stationary variable $\alpha$ is approximated in the same way, i.e.
\begin{equation}
 \label{eq:alFV}
 \alpha_K = \frac{1}{|K|} \int_K \alpha(x) \ \dm x.
\end{equation}
Finite volumes schemes we consider follow the general form, for all $K\in\T$ and $n\in\N$,
\begin{equation}
 \label{eq:FV}
 u^\np_K = u_K^n - \frac{\Delta t}{|K|} \sum_{L\in\Ns(K)} |e_{KL}| g(w_K^n,w_L^n;n_{KL})
\end{equation}
where $w_K^n=(u_K^n,\alpha_K)$ and $g$ is a numerical flux, which fulfills assumptions provided in the sequel.

The design of finite volume schemes for non-conservative equations is a very difficult task and convergence is hard to obtain in the singular case. In the general case, this can be achieved using random sampling instead of classical average techniques --- Glimm \cite{glimm65:_solut} vs. Godunov \cite{godunov:schema} --- as explained for instance in \cite{CC:ncv} (see also references therein). Here, since the non-conservative products are located in standing discontinuities, one may hope to have a better control of their approximation \cite{gosse01}. 

\begin{remark}
Using the theory of germs developed in \cite{AKR}, it is possible to prove convergence of well adapted numerical schemes in the resonant scalar case
\begin{equation*}
 \dv_t u + \dv_x (u^2/2) + u \dv_x \alpha = 0,
\end{equation*}
see \cite{AS:singular}. To do so, entropy solutions are defined relatively to particular stationary solutions of the equation, represented by piecewise constant functions, with discontinuities where $\alpha$ jumps corresponding to the non-conservative product. In fact, the present work may be seen as a continuation of \cite{AS:singular}, in the case of systems.
\end{remark}

Let us go back to the numerical flux $g$, which is a function from $(\Omega\times\R)^2\times S^{d-1}$ to $\R^N$. First, let us define the following function:
\begin{equation}
 \label{eq:Ucvx}
 \begin{aligned}
  \U \colon (\Omega\times\R)^2\times\R\times\R^+ & \longrightarrow \R^N \\
  (w_K,w_L;n,\nu) &\longmapsto u_K - \nu \big( g(w_K,w_L;n) - f(w_K) \bdot n \big)
 \end{aligned}
\end{equation}
which enables to rewrite the numerical scheme~\eqref{eq:FV} as the convex combination
\begin{equation}
 \label{eq:FVcvx}
 u_K^\np = \sum_{L\in\Ns(K)} \frac{|e_{KL}|}{|\dv K|} \U\bigg(w_K^n,w_L^n;n_{KL},\frac{|K|}{\Delta t |\dv K|} \bigg) .
\end{equation}
The numerical flux is subject to the following requirements:
\begin{enumerate}[label=(F\arabic*)]
\item \label{item:Fcons} \textbf{Consistency}. For all $w\in\Omega\times\R$ and all $n\in S^{d-1}$,
\begin{equation*}
 g(w,w;n) = f(w) \bdot n .
\end{equation*}
\item \label{item:Fcv} \textbf{Conservation}. For all $k=1,\dots,N$ such that $s^{(k)}\equiv0$, then for all $u_K,u_L\in\Omega$ and $n\in S^{d-1}$, $g(w_K,w_L;n)^{(k)}=-g(w_L,w_K;-n)^{(k)}$.
\item \label{item:Fpos} \textbf{Admissibility preservation}. There exists $L_g>0$ such that, for any $\nu \leq L_g^{-1}$, we have for all $w_K,w_L\in\Omega\times\R$ and all $n\in S^{d-1}$, 
\begin{equation*}
  \U(w_K,w_L;n,\nu) \in \Omega .
\end{equation*}
\item \label{item:Fent} \textbf{Entropy stability}. There exists an \emph{numerical entropy flux} $G$, from $(\Omega\times\R)^2\times S^{d-1}$ to $\R$, which is conservative, i.e. for all $w_K,w_L\in\Omega\times\R$ and all $n\in S^{d-1}$, 
\begin{equation*}
 G(w_K,w_L;n) = - G(w_L,w_K;-n)
\end{equation*}
and satisfies for any $\nu \leq L_g^{-1}$, all $w_K,w_L\in\Omega\times\R$ and all $n\in S^{d-1}$, 
\begin{equation}
\label{eq:entint}
 \eta(\U(w_K,w_L;n,\nu),\alpha_K) - \eta(w_K) + \nu ( G(w_K,w_L;n) - F(w_K) \bdot n ) \leq 0 .
\end{equation}
\item \label{item:Fwb} \textbf{Well-balancing for stationary states $\S$}. Let $\S$ some subset of $\Omega\times\R$. For all $w_K,w_L\in\S$ and all $n\in S^{d-1}$,
\begin{equation*}
 g(w_K,w_L;n) = f(w_K) \bdot n .
\end{equation*}
\end{enumerate}

The consistency condition \ref{item:Fcons} is very classical, and also ensures that the numerical entropy flux $G$ is consistent with the entropy flux $F$. Condition \ref{item:Fcv} allows to have the discrete conservation for the components of $u$ which satisfy a conservation law. From assumption \ref{item:Fpos} and the convex combination \eqref{eq:FVcvx}, one may deduce 
\begin{equation*}
 (u_K^0)_{K\in\T} \subset \Omega \Longrightarrow \forall n\in\N, \  (u_K^n)_{K\in\T} \subset \Omega,
\end{equation*}
under the Courant--Friedrichs--Lewy (CFL) condition 
\begin{equation}
 \label{eq:CFL}
 \Delta t \leq \inf_{K\in\T} \frac{|K|}{L_g|\dv K|} .
\end{equation}
Condition \ref{item:Fent} leads, under the same CFL condition, to the entropy inequality
\begin{equation}
 \label{eq:dei}
 \eta(w_K^\np) \leq \eta(w_K^n) - \frac{\Delta t}{|K|}  \sum_{L\in\Ns(K)} |e_{KL}| G(w_K^n,w_L^n;n_{KL})
\end{equation}
using the Jensen's inequality after applying $\eta$ to the convex combination \eqref{eq:FVcvx}. The entropy condition \ref{item:Fent} comes from the fundamental work \cite{HLL}. 

The well-balancing condition \ref{item:Fwb} directly yields
\begin{equation*}
 (u_K^0)_{K\in\T} \subset \S \Longrightarrow \forall n\in\N, \forall K\in\T, \ u_K^n=u_K^0 .
\end{equation*}
Well-balanced schemes have been introduced in \cite{greenberg96} and have been successfully developed by many authors, see for instance the books \cite{bouchutbook} and \cite{gossebook}. 

\subsection{An example of well-balanced scheme}

There exists a huge number of well-balanced schemes in the literature. However, very few satisfy conditions \ref{item:Fcons}--\ref{item:Fwb}, most of the authors only concentrate on \ref{item:Fwb} (and also \ref{item:Fcons} which is straightforward). In particular, condition \ref{item:Fent} may be hard to obtain. One may mention some of them: the non-conservative Godunov scheme \cite{greenberg96}, a modified kinetic scheme \cite{perthamesimeoni}, Suliciu's relaxation method \cite{bouchutbook,css,MR3454365}, entropy-stable schemes \cite{MR2799526}\dots 

Let us present the basic idea from \cite{greenberg96} in the one-dimensional case to construct well-balanced schemes (see also \cite{MR1885615} for a general presentation). First of all, let us recall that, for systems of conservation laws, the Godunov scheme can be interpreted as a two-step method, starting with an initial condition constant in each control volume: in the first step the Cauchy problem is exactly solved, and in the second step, a time step is chosen before any wave interaction and the exact solution is replaced by its mean in each control volume. By the divergence theorem and self-similarity of the solution at each interface, one recovers a finite volume formulation. Here, we follow the same reasoning. 

Consider a space step $h>0$ and an associated one-dimensional uniform mesh $\T=\cup_{i\in\Z}K_i$, with $K_i=(x_\imm,x_\ipp)$ and $x_\ipp=(\ipp)h$. A well-balanced scheme can be constructed as follows. Assume that $(u_i^n,\alpha_i)_{i\in\Z}$ are known:
\begin{enumerate}
 \item Solve the system for $t>0$ and $x\in\R$
\begin{equation*}
  \dv_t u + \dv_x f(u,\alpha) + s(u,\alpha) \alpha_h'(x) = 0 , \\
\end{equation*}
with data
\begin{equation*}
\begin{cases}
 \alpha_h(x) = \sum_{i\in\Z} \alpha_i \un_{K}(x) \\
 u_0(x) = \sum_{i\in\Z} u_i^n \un_{K}(x) 
\end{cases} .
\end{equation*}
We note $u_h(t,x)$ the exact solution. Remark that, at each interface $x_\ipp$, the solution is self-similar since we have locally a Riemann problem.
\item Choose $\Delta t$ such that the waves of each Riemann problem do not interact and apply the classical cell average:
\begin{equation*}
 u_i^\np = \frac{1}{h} \int_{K_i} u_h(\Delta t,x) \dm x.
\end{equation*}
\end{enumerate}
Let us note $u_\ipp(x/t)$ the solution of the Riemann problem at interface $x_\ipp$, with data $(u_i^n,\alpha_i)$ and $(u_\ip^n,\alpha_i)$. Since $\alpha_h$ is constant inside control volumes, one may apply the divergence theorem to get:
\begin{equation}
\label{eq:wb1d}
 u_i^\np = u_i^n - \frac{\Delta t}{h} \big[ f\big(u_\ipp(x/t=0^-)\big) - f\big(u_\imm(x/t=0^+)\big)\big] .
\end{equation}
As mentioned above, in general $f\big(u_\ipp(0^-)\big)\neq f\big(u_\ipp(0^+)\big)$ due to the non-conservative contribution of the source term, localized at each interface (see for instance \cite{MR1885615} for more details). On the other hand, assuming that the solution $u_\ipp$ is admissible and entropy satisfying, then the numerical flux also satisfies assumption \ref{item:Fcons} to \ref{item:Fwb}. It is worth noting that the calculation in step 1 of $u_h$ may be difficult. The extension to the multidimensional case on unstructured meshes is straightforward by extension of the finite volume form of the scheme \eqref{eq:wb1d}. We refer to \cite{chinnayya04} for more details, in the context of the shallow-water equations with bathymetry. 

\subsection{Numerical dissipation and relative entropy}

We now focus on the discrete version of the stability result stated in theorem~\ref{thm:stab}. To begin, let us give some details on the entropy dissipation of numerical schemes. For the study of time-continuous schemes, Tadmor introduced in \cite{MR890255} the function
\begin{equation}
\label{eq:gamma}
 \Gamma(w_{K},w_{L};n) = F(w_K) \bdot n + \dv_{u} \eta(w_{K}) \cdot \big( g(w_{K},w_{L};n) - f(w_{K}) \bdot n \big) .
\end{equation}
From assumption \ref{item:Fent} and following \cite{CMS}, one can prove:
\begin{lemma}
 \label{lem:flux}
 For all $(w_K,w_L)\in(\Omega\times\R)^2$ and all $n\in S^{d-1}$, we have for all $\nu\leq L_g^{-1}$
\begin{equation}
  \label{eq:dissflux}
  \Gamma(w_{K},w_{L};n) - G(w_{K},w_{L};n) \geq \frac{\nu\underline\eta}{2} \big| g(w_{K},w_{L};n) - f(w_{K}) \bdot n \big|^2 .
\end{equation}
\end{lemma}
\begin{proof}
We inject the definition \eqref{eq:gamma} of $\Gamma$ in the entropy flux inequality \eqref{eq:entint}, and obtain
\begin{multline*}
  \eta(\U(w_K,w_L;n,\nu),\alpha_K) - \eta(w_K) \\ + \nu \big( \dv_{u} \eta(w_{K}) \cdot ( g(w_{K},w_{L};n) - f(w_{K}) \bdot n ) \big) \\\leq \nu ( \Gamma(w_K,w_L;n) - G(w_{K},w_{L};n) ) .
\end{multline*}
By definition~\eqref{eq:Ucvx} of $\U$ and using the strict convexity of $\eta$ w.r.t. its first variable, see \ref{item:H1}, it results
\begin{equation*}
  \frac{\nu^2 \underline\eta}{2} | g(w_{K},w_{L};n) - f(w_{K}) \bdot n |^2 \leq \nu ( \Gamma(w_K,w_L;n) - G(w_{K},w_{L};n) ) ,
\end{equation*}
which is exactly the expected inequality.
\end{proof}
We are now in position to measure the numerical dissipation of entropy satisfying finite volume schemes:
\begin{proposition}
\label{prop:numdiss}
Consider a finite volume scheme \eqrefs{eq:icFV}{eq:FV} with a numerical flux which satisfies assumptions from \ref{item:Fcons} to \ref{item:Fent}. If there exists $\zeta\in(0,1)$ such that
\begin{equation}
\label{eq:CFLz}
\Delta t \leq (1-\zeta) \frac{\underline\eta}{\bar\eta} \frac{a^2 h}{L_g} ,
\end{equation}
then the approximate solution satisfies the discrete entropy inequality
\begin{multline}
\label{eq:numdiss}
 \eta(w_K^\np) - \eta(w_K^n) + \frac{\Delta t}{|K|} \sum_{L\in\Ns(K)} |e_{KL}| G(w_K^n,w_L^n;n_{KL}) \\
 \leq -\zeta \frac{\underline\eta\Delta t}{2|K|L_g} \sum_{L\in\Ns(K)} |e_{KL}| |g(w_K^n,w_L^n;n_{KL}) - f(w_K^n) \bdot n_{KL}|^{2} .
\end{multline}
\end{proposition}

\begin{proof}
First, let us remark that the upper bound \eqref{eq:eta2} on the spectral radius of the Hessian of $\eta$ w.r.t. $u$ leads to inequality
\begin{equation*}
 \eta(u_K^\np,\alpha_K) - \eta(u_K^n,\alpha_K) - \dv_{u}\eta(u_K^n,\alpha_K) \cdot (u_K^\np - u_K^n) \leq \frac{\bar\eta}{2} |u_K^\np-u_K^n|^2,
\end{equation*}
which, using the numerical scheme~\eqref{eq:FV}, yields
\begin{multline*}
 \eta(w_K^\np) - \eta(w_K^n) + \frac{\Delta t}{|K|} \sum_{L\in\Ns(K)} |e_{KL}| \dv_{u}\eta(w_K^{n}) \cdot g(w_K^n,w_L^n;n_{KL}) \\
 \leq \frac{\bar\eta}{2} \frac{\Delta t^2}{|K|^2} \sum_{L\in\Ns(K)} |e_{KL}|^2 |g(w_K^n,w_L^n;n_{KL}) - f(w_K^n) \bdot n_{KL}|^{2}.
\end{multline*}
Moreover, by definition of $\Gamma$ and from the divergence theorem, it results
\begin{equation*}
  \sum_{L\in\Ns(K)} |e_{KL}| \dv_{u}\eta(w_K^{n}) \cdot g(w_K^n,w_L^n;n_{KL}) = \sum_{L\in\Ns(K)} |e_{KL}| \cdot \Gamma(w_K^n,w_L^n;n_{KL}) ,
\end{equation*}
and thus, by lemma \ref{lem:flux}, the previous inequality becomes
\begin{multline*}
 \eta(w_K^\np) - \eta(w_K^n) + \frac{\Delta t}{|K|} \sum_{L\in\Ns(K)} |e_{KL}| \big( G(w_K^n,w_L^n;n_{KL}) \\
 + \frac{\underline\eta}{2L_g} |g(w_K^n,w_L^n;n_{KL}) - f(w_K^n)\bdot n_{KL}|^2 \big) \\
 \leq \frac{\bar\eta}{2} \frac{\Delta t^2}{|K|^2} \sum_{L\in\Ns(K)} |e_{KL}|^2 |g(w_K^n,w_L^n;n_{KL}) - f(w_K^n) \bdot n_{KL}|^{2}.
\end{multline*}
Using the isoperimetric assumption on the mesh \eqref{eq:meshreg}, one obtains successively
\begin{align*}
 &\eta(w_K^\np) - \eta(w_K^n) + \frac{\Delta t}{|K|} \sum_{L\in\Ns(K)} |e_{KL}| G(w_K^n,w_L^n;n_{KL}) \\
 &\leq \frac{\Delta t}{2|K|} \bigg(\frac{\bar\eta\Delta t}{|K|}\frac{h^{d-1}}{a} - \frac{\underline\eta}{2L_g} \bigg) \sum_{L\in\Ns(K)} |e_{KL}| |g(w_K^n,w_L^n;n_{KL}) - f(w_K^n) \bdot n_{KL}|^{2} \\
 &\leq \frac{\Delta t}{2|K|} \bigg(\frac{\bar\eta\Delta t}{a^2h} - \frac{\underline\eta}{L_g} \bigg) \sum_{L\in\Ns(K)} |e_{KL}| |g(w_K^n,w_L^n;n_{KL}) - f(w_K^n) \bdot n_{KL}|^{2} ,
\end{align*}
which, by the strengthened CFL condition \eqref{eq:CFLz}, provides inequality \eqref{eq:numdiss}.
\end{proof}
Inequality~\eqref{eq:numdiss} includes an lower bound for the numerical dissipation, which necessitates the use the CFL condition~\eqref{eq:CFLz}, which is strictly more restrictive than~\eqref{eq:CFL}. Note that it has been obtained without assuming the well-balanced property \ref{item:Fwb}. With this property, we obtain:
\begin{theorem}
  \label{thm:stabnum}
Let $H_0\in\R^N$ and consider the set $\S(H_0)$ defined by \ref{item:S1} and \ref{item:S2}, assumed to be nonempty. Consider $(\alpha_K)_{K\in\T}\subset\R$ and $(v_K)_{K\in\T}\subset\Omega$ such that for all $K\in\T$, $(v_K,\alpha_K)\in\S(H_0)$. Assume that the finite volume scheme \eqrefs{eq:icFV}{eq:FV} is defined by a numerical flux which satisfies assumptions from \ref{item:Fcons} to \ref{item:Fent}, and \ref{item:Fwb} related to $\S(H_0)$, and that the CFL stability condition~\eqref{eq:CFLz} holds true. Then, for any $u_0\in\BV(\R^d,\Omega)^N$, one has for all $K\in\T$ and $n\in\N$
\begin{multline}
\label{eq:numdissh}
 h(u_K^\np,v_K,\alpha_K) - h(u_K^n,v_K,\alpha_K) \\
 + \frac{\Delta t}{|K|} \sum_{L\in\Ns(K)} |e_{KL}| \G_{H_0}((u_K^n,\alpha_K),(u_L^n,\alpha_L);n_{KL}) \leq \D_K^n ,
\end{multline}
where $\D_K^n$ in the right-hand side of~\eqref{eq:numdiss}, $h$ the relative entropy introduced in definition~\ref{def:er} and the numerical flux $\G_{H_0}$ is given by
\begin{equation}
 \G_{H_0}(w_K,w_L;n) = G(w_K,w_L;n) - H_0 \cdot g(w_K,w_L;n)
\end{equation}
and is conservative: $\G_{H_0}(w_K,w_L;n) = -\G_{H_0}(w_L,w_K;-n)$.
\end{theorem}

\begin{proof}
 The proof is straightforward, using proposition~\ref{prop:numdiss}. Indeed, one has
\begin{align*}
 & h(u_K^\np,v_K,\alpha_K) - h(u_K^n,v_K,\alpha_K) \\ 
 &= \eta(u_K^\np,\alpha_K) - \eta(u_K^n,\alpha_K) - \dv_u \eta(v_K,\alpha_K) \cdot ( u_K^\np - u_K^n ) \\
 &\leq - \frac{\Delta t}{|K|} \sum_{L\in\Ns(K)} |e_{KL}| G((u_K^n,\alpha_K),(u_L^n,\alpha_L);n_{KL}) + \D_K^n \\
 &\quad + \frac{\Delta t}{|K|} \sum_{L\in\Ns(K)} |e_{KL}| H_0 \cdot g((u_K^n,\alpha_K),(u_L^n,\alpha_L);n_{KL}) 
\end{align*}
which exactly is~\eqref{eq:numdiss}. The conservative property of $\G_{H_0}$ is due to the conservative property of $G$ and the combination of assumptions~\ref{item:S2} and~\ref{item:Fcv}.
\end{proof}
A straightforward corollary of this theorem is a discrete version of the nonlinear stability theorem~\ref{thm:stab}:
\begin{corollary}
 Under the same notations and assumptions as in theorem~\ref{thm:stabnum}, if, for some $n\in\N^*$, $(u_K^n)_{K\in\T}\not\subset\S(H_0)$, then
\begin{equation}
\label{eq:numdissinth}
 \sum_{K\in\T} |K| h(u_K^\np,v_K,\alpha_K) <  \sum_{K\in\T} |K| h(u_K^n,v_K,\alpha_K) .
\end{equation}
\end{corollary}

It is worth noting that inequality~\eqref{eq:numdissinth} is strict, contrary to~\eqref{eq:stabt}, this is due to the numerical dissipation, represented by $\D_K^n$. Besides, it is important to note that the cancellation of the dissipation term $\D_K^n$ is related to the well-balancing property~\ref{item:Fwb}, which may lead to the following result of asymptotic stability:

\begin{corollary}
\label{cor:stdec}
 Let $(\alpha_K)_{K\in\T}\subset\R$ and $(u_K^0)_{K\in\T}\subset\Omega$ be given. Assume that there exists $H_0\in\R^N$ for which the set $\S(H_0)$ defined by \ref{item:S1} and \ref{item:S2} is nonempty and such that there exists a unique $(v_K)_{K\in\T}\subset\Omega$ satisfying $(v_K,\alpha_K)\in\S(H_0)$ for all $K\in\T$, and
\begin{equation}
\label{eq:compu0v}
 \sum_{K\in\T} |K| (v_K)^{(k)} = \sum_{K\in\T} |K| (u^0_K)^{(k)}
\end{equation}
for all component $k=1,\dots,N$ for which $s^{(k)}\equiv0$. \\
 Consider a finite volume scheme \eqrefs{eq:icFV}{eq:FV} defined by a numerical flux which satisfies assumptions from \ref{item:Fcons} to \ref{item:Fent}. Besides, we assume the well-balancing property \ref{item:Fwb}, but also its converse: let $(\bar u_K)_{K\in\T}\subset\Omega$, we assume that
\begin{equation}
\label{eq:vKuniq}
\begin{aligned}
& \forall K,L\in\T, \quad
g((\bar u_K,\alpha_K),(\bar u_L,\alpha_L);n) = f(\bar u_K,\alpha_K) \\
\Longrightarrow \quad& \forall K\in\T, \quad \bar u_K=v_K .
\end{aligned}
\end{equation}
Then, under the CFL stability condition~\eqref{eq:CFLz}, we have
\begin{equation}
 \label{eq:convnum}
 \lim_{n\to\infty} \sum_{K\in\T} |K| h(u_K^n,v_K,\alpha_K) = 0 .
\end{equation} 
In other words, the approximate solution provided by the finite volume \eqrefs{eq:icFV}{eq:FV} tends uniformly to the approximate stationary solution $(v_K)_{k\in\T}$ when $n$ tends to $+\infty$.
\end{corollary}
\begin{proof}
This proof consists in proving that 
\begin{equation*}
 V \colon (u_K^n)_{K\in\T} \longmapsto \sum_{K\in\T} |K| h(u_K^n,v_K,\alpha_K)
\end{equation*}
is a Lyapunov functional for the numerical scheme, relative to the stationary state $(v_K)_{K\in\T}$. According to condition~\eqref{eq:vKuniq}, the stationary state $(v_K)_K$ is the only fixed point of the numerical scheme. Moreover, thanks to \eqref{eq:compu0v}, the state $(v_K)_{K\in\T}$ can be attained from $(u_0)_K$, the numerical scheme being conservative for these components $k$. To conclude, it is sufficient to apply corollary~\ref{cor:stdec}, and the convexity property of the relative entropy $h$ stated in lemma~\ref{lem:h}.
\end{proof}

\subsection{An example of numerical asymptotic stability}

Let us provide a concrete application of the latter corollary. Let us go back to the two-dimensional shallow-water equations with bathymetry~\eqref{eq:stv}. In order to be in a configuration with only one possible ``lake at rest" stationary state~\eqref{eq:lake}, let us pose the equations~\eqref{eq:stv} in a (polygonal) bounded domain $D\subset\R^2$, with wall boundary conditions:
\begin{equation}
 \label{eq:bcstv}
 \forall t>0, x\in \dv D, \quad (hU)(t,x)\bdot n(x) = 0,
\end{equation}
where $n$ is unit normal to $\dv D$, outward t $\Omega$. Moreover, still to obtain the uniqueness of the stationary state, we assume that the bottom $\alpha$ and the initial data $u_0$ comply with
\begin{equation}
\label{eq:V0}
 V_0=\int_D h_0(x) \ \dm x > \max_{D} \alpha(x) - \int_D \alpha(x) \ \dm x .
\end{equation}
In other words, the total volume of water is sufficient to avoid the appearance of dry areas (using the conservation law satisfied by $h$). Indeed the case of non uniqueness could appear with the occurrence of at least two disjoint lakes with possible different surface levels. 

Let us now detail the numerical scheme. We assume that the boundary conditions~\eqref{eq:bcstv} are approximated by the mirror technique: for each boundary cell, a fictitious symmetric cell is created outside the domain $D$, with the same height of water and bathymetry, and with an opposite velocity, see for instance \cite{torobook}. This method ensures the conservation of the height of water and a good approximation of~\eqref{eq:bcstv}.

\begin{corollary}[Corollary~\ref{cor:stdec} rephrased for shallow-water equations]
Consider a bathymetry $\alpha$ and an initial data which satisfy~\eqref{eq:V0}. Let $Z_0$ the associated stationary surface level, defined by
\begin{equation*}
Z_0 = \frac{1}{|D|} \int_D (h_0+\alpha) \ \dm x .
\end{equation*}
The discrete bathymetry $(\alpha_K)_{K\in\T}$ being given by \eqref{eq:alFV}, the associated stationary state $(v_K)_{K\in\T}$ is uniquely defined by~\eqref{eq:lake}. Then, under the assumptions of corollary~\ref{cor:stdec} on the finite volume scheme~\eqrefs{eq:icFV}{eq:FV}, the associated approximate solution converges towards this stationary state, i.e.
\begin{equation}
 \label{eq:stvlim}
\begin{aligned}
 (h^n_K+\alpha_K)_K &\xrightarrow[n\to+\infty]{} Z_0, \\
 (U_K^n)_K &\xrightarrow[n\to+\infty]{} 0.
\end{aligned}
\end{equation}
\end{corollary}

\begin{remark}
 As far as the (entropy weak) solution of the shallow-water equations is considered, this \emph{asymptotic} stability for the ``lake at rest'' stationary state $(h+\alpha=Z_0,U=0)$ should fail. Indeed, different stationary states could exist but also time periodic non-dissipative smooth solutions. We can only obtain from theorem~\ref{thm:stab} the (non-asymptotic) stability of the stationary state $(h+\alpha=Z_0,U=0)$.
\end{remark}

\section{Some concluding remarks}
\label{sec:conc}

In this work, we have been able to compare entropy weak solutions to some stationary solutions. This analysis holds independently of the space dimension, the definition of the non-conservative products, the hyperbolicity, and the smallness and the smoothness of the solutions. These advantages are due to the use of the relative entropy, see also for instance \cite{MR3519973} and references therein. However, this analysis does not apply to every stationary states of interest. For instance, one-dimensional stationary states with a non-zero discharge of the shallow-water equations are not included, as well as transonic steady shock waves in a nozzle.

Concerning the numerical part, many numerical well-balanced schemes are not entropy-stable. The analysis we provide fails in this case, but could be adapted if the discrete entropy inequalities can be obtained up to some error terms (which in general require some smoothness on $\alpha$; see for instance \cite{MR3454365}). Note that the well-balancing property is also crucial to deduce the inequality satisfied by the relative entropy \eqref{eq:numdissh}.

We present at the end an application of (asymptotic) stability for the numerical approximation. In the same way, other applications can be obtained, replacing for instance wall boundary conditions by periodic boundary conditions, or using systems presented in section \ref{sec:ex}. 


\bibliographystyle{plain}

\end{document}